\documentclass{amsart}

\usepackage{amssymb, latexsym, MnSymbol, amscd ,graphicx,psfrag}
\usepackage[mathscr]{euscript}

\usepackage{CJKutf8}

\newcommand{\bC}{ {\mathbb{C}} }

\newcommand{\bL}{\mathbb{L}}

\newcommand{\bP}{\mathbb{P}}
\newcommand{\bQ}{\mathbb{Q}}
\newcommand{\bR}{\mathbb{R}}

\newcommand{\bZ}{\mathbb{Z}}

\newcommand{\cL}{\mathcal{L}}
\newcommand{\cM}{\mathcal{M}}

\newcommand{\cX}{\mathcal{X}}

\newcommand{\Aut}{\mathrm{Aut}}

\newcommand{\Hom}{\mathrm{Hom}}

\newcommand{\ev}{\mathrm{ev}}
\newcommand{\val}{ {\mathrm{val}} }
\newcommand{\vir}{ {\mathrm{vir}} }

\newcommand{\NE}{{\mathrm{NE}}}

\newcommand{\sh}{\mathsf{h}}

\newcommand{\sw}{\mathsf{w}}

\newcommand{\tphi}{\widetilde{\phi}}

\newcommand{\tC}{\widetilde{C}}

\newcommand{\tV}{\widetilde{V}}

\newcommand{\tX}{\widetilde{X}}

\newcommand{\vGa}{\vec{\Gamma}}
\newcommand{\vd}{\vec{d}}
\newcommand{\vf}{\vec{f}}

\newcommand{\vs}{\vec{s}}

\newcommand{\Mbar}{\overline{\cM}}

\newtheorem{dummy}{dummy}[section]
\newtheorem{lemma}[dummy]{Lemma}
\newtheorem{theorem}[dummy]{Theorem}

\newtheorem{proposition}[dummy]{Proposition}

\newtheorem{definition}[dummy]{Definition}

\begin{document}
	\begin{CJK}{UTF8}{gbsn}
	\title{Open/closed correspondence for the projective line}

	\author{Zhengyu Zong}
	\address{Zhengyu Zong, Department of Mathematical Sciences,
		Tsinghua University, Haidian District, Beijing 100084, China}
	\email{zyzong@mail.tsinghua.edu.cn}

	\begin{abstract}
		We establish a correspondence between the disk invariants of the complex projective line $\bP^1$ with boundary condition specified by an $S^1$-invariant Lagrangian sub-manifold $L$ and the genus-zero closed Gromov-Witten invariants of a toric surface $X$.
	\end{abstract}
	\maketitle
\setcounter{tocdepth}{1}
\tableofcontents

\section{Introduction}

\subsection{Historical background and motivation}

\subsubsection{Open/closed correspondence for Calabi-Yau 3-folds}
The \emph{open/closed correspondence}, proposed by Mayr \cite{Mayr01} as a class of open/closed string dualities and developed by Lerche-Mayr \cite{LM01}, is a conjectural relation between the topological
amplitudes at genus zero of an open string geometry on a Calabi-Yau 3-fold $X$ relative to a Lagrangian $L$ and a closed string geometry on a corresponding Calabi-Yau 4-fold ̃$\tX$. In terms of Gromov-Witten theory, the correspondence conjecturally relates the disk Gromov-Witten invariants of the Calabi-Yau 3-fold $(X,L)$ to the genus-zero closed Gromov-Witten invariants of the dual Calabi-Yau 4-fold $\tX$.

In \cite{LY21}, the open/closed correspondence is mathematically proved for the case of a toric Calabi-Yau 3-fold $X$ and a Lagrangian submanifold $L$ of Aganagic-Vafa type. In \cite{LY22}, the above result is generalized to the case of a toric Calabi-Yau 3-orbifold $\cX$ and a Lagrangian suborbifold $\cL$ of Aganagic-Vafa type. The open/closed correspondence is also proved for the quintic threefold in \cite{AL23} in terms of Gauged Linear Sigma Model. By the open/relative correspondence for toric Calabi-Yau 3-orbifolds in \cite{FLT12}, the open/closed correspondence for toric Calabi-Yau 3-orbifolds can also be viewed as the log-local correspondence \cite{vGGR19}. See e.g. \cite{BBvG20,BBvG20b} for related works.

\subsubsection{Open/closed correspondence for the projective line}
In this paper, we show that the open/closed correspondence also works for the complex projective line $\bP^1$, although $\bP^1$ is not Calabi-Yau. 

Let $ [z_1, z_2]$ be the homogeneous coordinates of $\bP^1$. Let $S^1$ act on $\bP^1$ by $ t\cdot [z_1, z_2]=[ z_1, t z_2]$. Let $L:=\{[1,z]\in\bP^1\mid |z|=1\}$, which is a Lagrangian sub-manifold of $\bP^1$. The above $S^1$ action preserves $L$. By taking a M\"{o}bius transform, we can identify the pair $(\bP^1,L)$ with $(\bP^1,\bR\bP^1)$. We will study the $S^1$-equivariant open Gromov-Witten theory of $(\bP^1,L)$ in Section \ref{sec:GWX}. The open Gromov-Witten theory with descendants of $(\bP^1,\bR\bP^1)$ is studied in \cite{BNPT22}. Related works can be found in \cite{BT17,Net17,PST14,Tes23}.

On the other hand, we will define a toric surface $X$ in Section \ref{sec:geometryX} and study the equivariant closed Gromov-Witten theory of $X$ in Section \ref{sec:GWX}. The main result (Theorem \ref{thm:correspondence}) relates the $S^1$-equivariant open Gromov-Witten invariants of $(\bP^1,L)$ to the equivariant closed Gromov-Witten invariants of $X$.

We emphasize the following two features of our main result:
\begin{enumerate}
	\item We include \emph{descendant} insertions in both open Gromov-Witten invariants of $(\bP^1,L)$ and closed Gromov-Witten invariants of $X$.\\
	\item Since $\bP^1$ and $L$ are compact, we can take the non-equivariant limit of the $S^1$-equivariant open Gromov-Witten invariants of $(\bP^1,L)$. This limit equals to the non-equivariant open Gromov-Witten invariants of $(\bP^1,L)$ studied in \cite{BNPT22} via symplectic geometry. This feature is different from the case of toric Calabi-Yau 3-folds, which are always non-compact.
\end{enumerate}

We hope the result in this paper can contribute to understanding of the open/closed correspondence for more general target spaces, especially for those beyond Calabi-Yau level.

\subsection{Statement of the main result}
Let $S^1$ act on $\bP^1$ by
$$
t\cdot [z_1, z_2]=
[ z_1, t z_2].
$$
Let $\bC[u]=H_{S^1}^*(\mathrm{point};\bC)$ be
the $S^1$-equivariant cohomology of a point. Let $p_1=[1,0]$ and $p_2=[0,1]$ be the $S^1$ fixed points.

The $S^1$-equivariant cohomology of $\bP^1$ is given by
$$
H^*_{S^1}(\bP^1;\bC)= \bC[H,u]/\langle (H-u)H\rangle
$$
where $\deg H=\deg u=2$. Let
$$
\phi_1 := H,\quad
\phi_2 := H-u
$$

Let
$$
L:=\{[1,z]\in\bP^1\mid |z|=1\},
$$ 
which is a Lagrangian sub-manifold of $\bP^1$ and is preserved by the above $S^1$ action. We identify the relative homology group $H_2(\bP^1,L;\bZ)=\bZ^2$ so that the two oriented hemispheres containing $p_1$ and $p_2$ represent the generators $(1,0)$ and $(0,1)$ respectively. Under this identification, let $E(\bP^1,L)=\bZ_{\geq 0}^2$.

Given nonnegative integer $n$ and an element
$\beta=(d_{+},d_-)\in E(\bP^1,L),d_+\neq d_-$, we will consider degree-$\beta$, $S^1$-equivariant disk \emph{descendant Gromov-Witten invariant} (see Section \ref{sec:GWP1} for more details)  
$$
\left\langle  \tau_{a_1}(\phi_{\alpha_1}),\cdots,\tau_{a_n}(\phi_{\alpha_n})\right\rangle ^{(\bP^1,L),S^1}_{\beta},\quad \alpha_1,\cdots,\alpha_n\in\{1,2\}, a_1,\cdots,a_n\in\bZ_{\geq 0}.
$$
Here we assume $d_+\neq d_-$ since when $d_+=d_-$ the disk invariants vanish, see \cite[Lemma 4.4]{BNPT22}.

Let $N=\bZ^2$ and define $v_1,v_2,v_3,v_4\in N$ as
$$
v_1=(0,1),\quad v_2=(1,0),\quad v_3=(-1,1),\quad v_4=(1,-1).
$$
Let $\tau_i=\bR_{\geq 0}v_i\subset N_{\bR}:=N\otimes\bR,i=1,2,3,4$ be the corresponding 1-dimensional cones. Define 2-dimensional cones $\sigma_0, \sigma_1,\sigma_2\subset N_{\bR}$ as
$$
\sigma_0=\bR_{\geq 0}v_1+\bR_{\geq 0}v_2,\quad \sigma_1=\bR_{\geq 0}v_1+\bR_{\geq 0}v_3,\quad \sigma_2=\bR_{\geq 0}v_2+\bR_{\geq 0}v_4.
$$
Let $\Sigma$ be the fan with top dimensional cones $\sigma_0, \sigma_1,\sigma_2$ and let $X$ be the toric surface defined by $\Sigma$. We refer to \cite{CLS11,Fulton93} for the general theory of
toric varieties.

The torus $T:=N\otimes\bC^*\cong (\bC^*)^2$ acts on $X$. Let $p_{\sigma_i}=V(\sigma_i), i=0,1,2$ be the $T-$fixed points and let $l_{\tau_i}=V(\tau_i), i=1,2,3,4$ be the $T-$invariant lines. Define $\tphi_1,\tphi_2$ in the $T-$equivariant cohomology of $X$ as
 $$\tphi_1:=\frac{[p_{\sigma_1}]}{-u_1-u_2},\tphi_2:=\frac{[p_{\sigma_2}]}{-u_1-u_2},
 $$
where $\bC[u_1,u_2]=H_{T}^*(\mathrm{point};\bC)$ is
the $T$-equivariant cohomology of a point. 

We have $H_2(X;\bZ)=\bZ l_{\tau_1}\oplus \bZ l_{\tau_2}$. So we identify $H_2(X;\bZ)$ to $\bZ^2$, where $(d_1,d_2)\in\bZ^2$ is identified to $d_1 l_{\tau_1}+d_2 l_{\tau_2}$. Let $\NE(X) \subset H_2(X;\bR)=H_2(X;\bR)$ be the Mori cone generated by effective curve classes in $X$, and $E(X)$ denote the semigroup $\NE(X)\cap H_2(X;\bZ)$. We identify the semigroup $E(X)$ to $\bZ_{\geq 0}^2$.

Given nonnegative integer $n$ and an element
$\beta\in E(X)$, we will consider genus-$0$, degree-$\beta$, $T$-equivariant \emph{descendant Gromov-Witten invariant} (see Section \ref{sec:GWX} for more details)
$$\left\langle\tau_{a_1}(\tphi_{\alpha_1}),\cdots,\tau_{a_n}(\tphi_{\alpha_n})\right\rangle ^{X,T}_{0,\beta},\quad \alpha_1,\cdots,\alpha_n\in\{1,2\}, a_1,\cdots,a_n\in\bZ_{\geq 0}.
$$

The following theorem is the main result of this paper:
\begin{theorem}[=Theorem \ref{thm:correspondence}]\label{thm:main}
	Let $\beta=(d_1,d_2)\in \bZ_{\geq 0}^2,d_1\neq d_2$. Then for $\alpha_1,\cdots,\alpha_n\in\{1,2\},a_1,\cdots,a_n\in\bZ_{\geq 0}$, we have
	\begin{eqnarray*}
		&&\left\langle  \tau_{a_1}(\tphi_{\alpha_1}),\cdots,\tau_{a_n}(\tphi_{\alpha_n})\right\rangle ^{X,T}_{0,\beta}\vert_{u_1+u_2=0,u_1=u}\\
		&=&\sum \frac{(-u^2)^{l+m-1}}{|\Aut(\mu^1,d^1,A^1)||\Aut(\mu^2,d^2,A^2)|}\prod_{i=1}^{l}\big(\frac{\mu^1_i}{-u}\big)^{1+b_i}\prod_{j=1}^{m}\big(\frac{\mu^2_j}{u}\big)^{1+c_j}\frac{(l+m-3)!}{\prod_{i=1}^{l}b_i!\prod_{j=1}^{m}c_j!}\\
		&&\cdot\prod_{i=1}^{l}\big(\frac{(-1)^{\mu^1_i}}{u}\left\langle  \prod_{k\in A^1_i}\tau_{a_k}(\phi_{\alpha_k})\right\rangle ^{(\bP^1,L),S^1}_{(0,1),\beta^1_i}\big)
		\prod_{j=1}^{m}\big(\frac{(-1)^{\mu^2_j+1}}{u}\left\langle  \prod_{k\in A^2_j}\tau_{a_k}(\phi_{\alpha_k})\right\rangle ^{(\bP^1,L),S^1}_{(0,1),\beta^2_j}\big).
	\end{eqnarray*}
	Here the sum is taken over $\{((\mu^1_i,d^1_i,A^1_i,b_i)_{1\leq i\leq l},(\mu^2_j,d^2_j,A^2_j,c_j)_{1\leq j\leq m},)\mid l,m\geq 0, l+m\geq 1,\mu^1_1\neq \mu^2_1 \textrm{ when } l=m=1, \mu^1_i,\mu^2_j>0, d^1_i,d^2_j,b_i,c_j\geq 0, \sum_{i=1}^{l}(d^1_i+\mu^1_i)+\sum_{j=1}^{m}d^2_j=d_1, \sum_{i=1}^{l}d^1_i+\sum_{j=1}^{m}(d^2_j+\mu^2_j)=d_2,\sqcup_{i=1}^{l}A^1_i\sqcup \sqcup_{j=1}^{m}A^2_j=\{1,\cdots,n\},\sum_{i=1}^{l}b_i+\sum_{j=1}^{m}c_j=l+m-3\}$ and $ \beta^1_i=(d^1_i+\mu^1_i,d^1_i), \beta^2_j=(d^2_j,d^2_j+\mu^2_j)$. The automorphism groups $\Aut(\mu^1,d^1,A^1)$ and $\Aut(\mu^2,d^2,A^2)$ are for the tuples $\big((\mu^1,d^1,A^1)_{1\leq i\leq l} \big)$ and $\big((\mu^2,d^2,A^2)_{1\leq j\leq m} \big)$ respectively.
	
\end{theorem}

\subsection{Overview of the paper}
In Section \ref{sec:geometry}, we review the open geometry of $(\bP^1,L)$ and the closed geometry of the toric surface $X$. We also study the equivariant cohomology of $\bP^1$ and $X$. In Section \ref{sec:GWP1}, we study the $S^1$-equivariant open Gromov-Witten theory of $(\bP^1,L)$ and give the graph sum formula via virtual localization. In Section \ref{sec:GWX}, we study the equivariant closed Gromov-Witten theory of $X$ and study the corresponding graph sum formula. In Section \ref{sec:correspondence}, we prove the open closed correspondence, which is the main theorem of this paper.

\subsection*{Acknowledgements}
The author would like to thank Song Yu for useful suggestions on the proof of the open/closed correspondence. The author would also like to thank Bohan Fang and Chiu-Chu Melissa Liu for useful discussions. The work of the author is partially supported by the Natural Science Foundation of Beijing, China grant No. 1252008 and NSFC grant No. 11701315.

\section{Geometric setup}\label{sec:geometry}

\subsection{Equivariant cohomology of $\bP^1$}\label{sec:geometryP1}

Let $S^1$ act on $\bP^1$ by
$$
t\cdot [z_1, z_2]=
[ z_1, t z_2].
$$
Let $\bC[u]=H_{S^1}^*(\mathrm{point};\bC)$ be
the $S^1$-equivariant cohomology of a point.

The $S^1$-equivariant cohomology of $\bP^1$ is given by
$$
H^*_{S^1}(\bP^1;\bC)= \bC[H,u]/\langle (H-u)H\rangle
$$
where $\deg H=\deg u=2$.
Let $p_1=[1,0]$ and $p_2=[0,1]$ be the $S^1$ fixed points.
Then $H|_{p_1}= u, H|_{p_2}= 0$. The $S^1$-equivariant Poincar\'{e} dual
of $p_1$ and $p_2$ are $H$ and $H-u$, respectively.
Let
$$
\phi_1 := H,\quad
\phi_2 := H-u
$$
Then $\deg \phi_\alpha=2$,
$$
\phi_1 \cup \phi_1 = u \phi_1,\quad \phi_2 \cup \phi_2 = -u \phi_2,\quad \phi_1 \cup \phi_2 = 0.
$$

Let
$$
L:=\{[1,z]\in\bP^1\mid |z|=1\},
$$ 
which is a Lagrangian sub-manifold of $\bP^1$. The above $S^1$ action preserves $L$. By taking a M\"{o}bius transform, we can identify the pair $(\bP^1,L)$ with $(\bP^1,\bR\bP^1)$.

We identify the relative homology group $H_2(\bP^1,L;\bZ)=\bZ^2$ so that the two oriented hemispheres containing $p_1$ and $p_2$ represent the generators $(1,0)$ and $(0,1)$ respectively. Under this identification, let $E(\bP^1,L)=\bZ_{\geq 0}^2$. Let $D$ be the disk and let $\partial D$ be its boundary. Given $\beta\in E(\bP^1,L)$, a \emph{degree $\beta$ disk map} to $(\bP^1,L)$ is a map $f:(D,\partial D)\to (\bP^1,L)$ with $f_*([D])=\beta\in E(\bP^1,L)\subset H_2(\bP^1,L;\bZ)$.

\subsection{Geometry and equivariant cohomology of $X$}\label{sec:geometryX}
In this subsection, we construct a toric surface $X$ and study its geometry and equivariant cohomology. We refer to \cite{CLS11,Fulton93} for the general theory of
toric varieties.

Let $N=\bZ^2$ and define $v_1,v_2,v_3,v_4\in N$ as
$$
v_1=(0,1),\quad v_2=(1,0),\quad v_3=(-1,1),\quad v_4=(1,-1).
$$
Let $\tau_i=\bR_{\geq 0}v_i\subset N_{\bR}:=N\otimes\bR,i=1,2,3,4$ be the corresponding 1-dimensional cones. Define 2-dimensional cones $\sigma_0, \sigma_1,\sigma_2\subset N_{\bR}$ as
$$
\sigma_0=\bR_{\geq 0}v_1+\bR_{\geq 0}v_2,\quad \sigma_1=\bR_{\geq 0}v_1+\bR_{\geq 0}v_3,\quad \sigma_2=\bR_{\geq 0}v_2+\bR_{\geq 0}v_4.
$$
Let $\Sigma$ be the fan with top dimensional cones $\sigma_0, \sigma_1,\sigma_2$ and let $X$ be the toric surface defined by $\Sigma$.

Let $T:=N\otimes\bC^*\cong (\bC^*)^2$ be the algebraic torus contained in $X$ as an open dense subset. The nature action of $T$ on itself extends to a $T$ action on $X$. Let $p_{\sigma_i}=V(\sigma_i), i=0,1,2$ and let $l_{\tau_i}=V(\tau_i), i=1,2,3,4$. Then $p_{\sigma_i}, i=0,1,2$ are the $T-$fixed points and $l_{\tau_i}, i=1,2,3,4$ are the $T-$invariant lines. Let $M:=\Hom(N,\bZ)$, which can be canonically identified as the character
lattice $\Hom(T,\bC^*)$ of $T$. For $\tau_i\subset\sigma_j$, let $\sw(\tau_i,\sigma_j)\in M$ be the weight of the $T$-action on $T_{p_{\sigma_j}}l_{\tau_i}$, the tangent line to $l_{\tau_i}$ at the fixed point $p_{\sigma_j}$. The weights $\sw(\tau_i,\sigma_j)$ are given by
$$
\sw(\tau_1,\sigma_0)=-u_1,\quad\sw(\tau_2,\sigma_0)=-u_2,\quad \sw(\tau_1,\sigma_1)=u_1,
$$
$$
\sw(\tau_3,\sigma_1)=-u_1-u_2,\quad \sw(\tau_2,\sigma_2)=u_2,\quad \sw(\tau_4,\sigma_0)=-u_1-u_2.
$$

Define $\tphi_1,\tphi_2,\tphi_0\in H^*_T(X;\bC)\otimes_{\bC[u_1,u_2]}\bC(u_1,u_2)$ as $$\tphi_1:=\frac{[p_{\sigma_1}]}{-u_1-u_2},\tphi_2:=\frac{[p_{\sigma_2}]}{-u_1-u_2},\tphi_0:=\frac{[p_{\sigma_0}]}{u_1u_2}.$$
Then $\tphi_1,\tphi_2,\tphi_0$ is a basis of $H^*_T(X;\bC)\otimes_{\bC[u_1,u_2]}\bC(u_1,u_2)$, $\deg\tphi_1=\deg\tphi_2=2,\deg\tphi_0=0$, and 
\begin{equation}\label{eqn:tphi}
\tphi_1 \cup \tphi_1 = u_1 \tphi_1,\quad \tphi_2 \cup \tphi_2 = u_2 \tphi_2,\quad \tphi_0 \cup \tphi_0 = \tphi_0,\quad \tphi_i \cup \tphi_j = 0,i\neq i.
\end{equation}

We have $H_2(X;\bZ)=\bZ l_{\tau_1}\oplus \bZ l_{\tau_2}$. So we identify $H_2(X;\bZ)$ to $\bZ^2$, where $(d_1,d_2)\in\bZ^2$ is identified to $d_1 l_{\tau_1}+d_2 l_{\tau_2}$. Let $\NE(X) \subset H_2(X;\bR)=H_2(X;\bR)$ be the Mori cone generated by effective curve classes in $X$, and $E(X)$ denote the semigroup $\NE(X)\cap H_2(X;\bZ)$. We identify the semigroup $E(X)$ to $\bZ_{\geq 0}^2$.

\section{Open Gromov-Witten theory of $\bP^1$}\label{sec:GWP1}
In this section, we review the $S^1$-equivariant open Gromov-Witten theory of $\bP^1$, which is studied in \cite{BNPT22}. We refer to \cite{Liu02} for open Gromov-Witten theory of more general target spaces with an $S^1$ action.
\subsection{Open Gromov-Witten invariants}
Given nonnegative integer $n$ and an element
$\beta=(d_{+},d_-)\in E(\bP^1,L),d_+\neq d_-$, let $\Mbar_{(0,1),n}(\bP^1,L, \beta)$ be the moduli space of degree-$\beta$ stable disk maps to $(\bP^1,L)$ with $n$ interior marked points. Let $\ev_i:\Mbar_{(0,1),n}(\bP^1,L, \beta)\to \bP^1$ be the evaluation map
at the $i$-th marked point. The $S^1$-action on $(\bP^1,L)$ induces
$S^1$-actions on the moduli space $\Mbar_{(0,1),n}(\bP^1,L, \beta)$, and
the evaluation map $\ev_i$ is $S^1$-equivariant.

For $i=1,\dots,n$, let $\bL_i$ be the $i$-th tautological line bundle over $\Mbar_{(0,1),n}(\bP^1,L, \beta)$ formed
by the cotangent line at the $i$-th marked point. Define the $i$-th descendant class $\psi_i$ as
$$
\psi_i := c_1(\bL_i)\in H^2(\Mbar_{(0,1),n}(\bP^1,L, \beta);\bQ).
$$
We choose an $S^1$-equivariant
lift $\psi_i^{T}\in H^2_{S^1}(\Mbar_{(0,1),n}(\bP^1,L, \beta);\bQ)$
of $\psi_i$. 

Let $\gamma_1,\dots, \gamma_n\in H_{S^1}^*(\bP^1,\bC)$ and $a_1,\dots,a_n\in \bZ_{\ge 0}$. We define the degree-$\beta$, $S^1$-equivariant disk \emph{descendant Gromov-Witten invariant}
\begin{align*}
	&\langle \tau_{a_1}(\gamma_1), \dots, \tau_{a_n}(\gamma_n)\rangle^{(\bP^1,L),S^1}_{(0,1),\beta} \\
	& := \int_{[\Mbar_{(0,1),n}(\bP^1,L, \beta)^{S^1}]} \frac{\iota^*\big(\prod_{i=1}^n \ev_i^*(\gamma_i)(\psi_i^{S^1})^{a_i}\big)}{e_{S^1}(N^\vir)}
	\quad \in \bC(u)
\end{align*}
where $\Mbar_{(0,1),n}(\bP^1,L, \beta)^{S^1}$ is the $S^1$-fixed locus, $e_{S^1}(N^\vir)$ is the $S^1$-equivariant Euler class of the virtual normal bundle of $\Mbar_{(0,1),n}(\bP^1,L, \beta)^{S^1}$ in $\Mbar_{(0,1),n}(\bP^1,L, \beta)$, and $\iota: \Mbar_{(0,1),n}(\bP^1,L, \beta)^{S^1}\hookrightarrow \Mbar_{(0,1),n}(\bP^1,L, \beta)$ is the inclusion map. Since $\Mbar_{(0,1),n}(\bP^1,L, \beta)^{S^1}$ is a compact orbifold without boundary, the above integral is well-defined.

By taking a M\"{o}bius transform, we can identify the pair $(\bP^1,L)$ with $(\bP^1,\bR\bP^1)$. By the main theorem of \cite{BNPT22}, the non-equivariant limit of 
$$
\left\langle  \tau_{a_1}(\phi_{\alpha_1}),\cdots,\tau_{a_n}(\phi_{\alpha_n})\right\rangle ^{(\bP^1,L),S^1}_{(0,1),\beta},\quad \alpha_i \in \{1,2\}
$$
is equal to the non-equivariant stationary open Gromov-Witten invariant of $(\bP^1,L)$ constructed in \cite{BNPT22}.

\subsection{Virtual localization formula}

\begin{definition}
	Let $n\in\bZ_{\geq 0}$ and $\beta=(d_{+},d_-)\in E(\bP^1,L),d_+\neq d_-$. A genus-zero, $n$-pointed, degree-$\beta$ decorated graph for $(\bP^1,L)$ is a tuple $\vGa=(\Gamma,\vf,\vd,\vs)$ where:
	\begin{itemize}
		\item $\Gamma$ is a tree. Let $\tV(\Gamma)$ denote the set of
		vertices in $\Gamma$, $E(\Gamma)$ denote the set of edges in $\Gamma$, and $F(\Gamma):=\{(e,v)\in E(\Gamma)\times \tV(\Gamma)\mid e\textrm{ is attached to } v \}$ be the set of flags. For each $v\in \tV(\Gamma)$, define $E_v:=\{e\in E(\Gamma)\mid (e,v)\in F(\Gamma)\}$ and let $\val(v):=|E_v|$ denote the \emph{valence} of $v$.
		\item We fix a univalent vertex $v_0\in \tV(\Gamma)$ called the \emph{root} of $\Gamma$ and let $V(\Gamma):=\tV(\Gamma)\setminus \{v_0\}$. Let $e_0\in E(\Gamma)$ be the unique edge attaching to $v_0$ and let $v_1\neq v_0$ be the other vertex where $e_0$ is attached.
		\item $\vf:V(\Gamma)\cup E(\Gamma)\to \{\sigma_+,\sigma_-\}\cup\{\tau\} $ is the \emph{label map} that sends each vertex $v\in V(\Gamma)$ to some $\sigma_v\in \{\sigma_+,\sigma_-\}$ and each edge $e\in E(\Gamma)$ to $\tau_e= \tau$. Let $\Upsilon_{\bP^1}$ be the graph with two vertices $\sigma_+,\sigma_-$ joined by an edge $\tau$. We require that $\vf$ defines a map from
		the graph $\Gamma$ to the graph $\Upsilon_{\bP^1}$. We also require that $\vf(v_1)=\sigma_+$ if $d_+>d_-$; and $\vf(v_1)=\sigma_-$ if $d_+<d_-$.
		\item $\vd: E(\Gamma)\to\bZ_{>0}$ is the \emph{degree map}. We denote $d_e:=\vd(e)$ for each $e\in E(\Gamma)$. We require that $\sum_{e\in E(\Gamma)}d_e=d_+,\sum_{e\in E(\Gamma)\setminus\{e_0\}}d_e=d_-$ if $d_+>d_-$; and $\sum_{e\in E(\Gamma)}d_e=d_-,\sum_{e\in E(\Gamma)\setminus\{e_0\}}d_e=d_+$ if $d_+<d_-$.
		\item $\vs:\{1,\cdots,n\}\to V(\Gamma)$ is the \emph{marking map}.
		
	\end{itemize}
	
\end{definition}
Let $ G_{n}(\bP^1,L,\beta)$ denote the set of all genus-zero, $n$-pointed, degree-$\beta$ decorated graphs for $(\bP^1,L)$. Given $\vGa\in G_{n}(\bP^1,L,\beta)$, we introduce the following notations:
\begin{itemize}
	\item For each $v\in V(\Gamma)$, define $S_v:=\vs^{-1}(v)$. 
	\item Let $\Aut(\vGa)$ denote the group of automorphisms of $\vGa$, i.e. automorphisms of the graph $\Gamma$ that make the label maps $\vf,\vd,\vs,v_0$ invariant.
\end{itemize}

For $d\in\bZ_{>0}$, we define
\begin{equation}
	\sh(\tau,d)=\frac{(-1)^d d^{2d}}{(d!)^2u_i^{2d}}.
\end{equation}

We define
\begin{equation}
	\sw(\sigma_{\pm})=\pm u
\end{equation}
 and 
\begin{equation}
	D_+(\mu)=\frac{\mu^{\mu}}{\mu ! u^{\mu}},\quad 	D_-(\mu)=(-1)^{\mu}\frac{\mu^{\mu}}{\mu ! u^{\mu}},\quad \mu \in\bZ_{>0}.
\end{equation}

The following proposition is the main theorem of \cite{BNPT22}. Similarly localization computation for open Gromov-Witten invariants can be found, for example, in \cite{KL01,FLT12}.
\begin{proposition}\label{prop:locP1}
	Let $\beta=(d_{+},d_-)\in E(\bP^1,L),d_+\neq d_-$ and $\mu=|d_+-d_-|$. When $d_+>d_-$, we have the following graph sum formula：
	\begin{eqnarray*}
		&&\left\langle  \tau_{a_1}(\phi_{\alpha_1}),\cdots,\tau_{a_n}(\phi_{\alpha_n})\right\rangle ^{(\bP^1,L),S^1}_{(0,1),\beta}\\
		&=&\sum_{\vGa\in G_{n}(\bP^1,L,\beta)}\frac{1}{|\Aut(\vGa)|}\prod_{e\in E(\Gamma)\setminus\{e_0\}}\frac{\sh(\tau,d_e)}{d_e}\prod_{v\in V(\Gamma)}\big(\sw(\sigma_v)^{\val(v)-1}\prod_{i\in S_v}i^*_{\sigma_v}\phi_{\alpha_i} \big)\\
		&&\cdot \prod_{v\in V(\Gamma)}\int_{\Mbar_{0,E_{v}\cup S_{v}}}\frac{\prod_{i\in S_v}\psi_i^{a_i}}{\prod_{e\in E_{v}}(w_{(e,v)}-\psi_{(e,v)})}D_+(\mu),
	\end{eqnarray*}
where $w_{(e,v)}=\pm \frac{u}{d_e}$ if $\vf(v)=\sigma_{\pm}$. When $d_+<d_-$, we have the same graph sum formula with $D_+(\mu)$ replaced by $D_-(\mu)$.
\end{proposition}
In Proposition \ref{prop:locP1}, we used the following convention for the unstable integrals:
\begin{equation}\label{eqn:unstable}
	\int_{\Mbar_{0,1}}\frac{1}{1-d\psi}=\frac{1}{d^2},\quad
	\int_{\Mbar_{0,2}}\frac{1}{(1-d_1\psi_1)(1
		-d_2\psi_2)}
	= \frac{1}{d_1+d_2},\quad
	\int_{\Mbar_{0,2}}\frac{1}{(1-d\psi_1)}=\frac{1}{d}.
\end{equation}

\section{Gromov-Witten theory of $X$}\label{sec:GWX}

\subsection{Equivariant Gromov-Witten invariants}
Given nonnegative integer $n$ and an effective curve class
$\beta\in E(X)$, let $\Mbar_{0,n}(X, \beta)$ be the moduli space of genus-$0$, $n$-pointed, degree-$\beta$ stable maps to $X$. Let $\ev_i:\Mbar_{0,n}(X,\beta)\to X$ be the evaluation map
at the $i$-th marked point. The $T$-action on $X$ induces
$T$-actions on the moduli space $\Mbar_{0,n}(X,\beta)$, and
the evaluation map $\ev_i$ is $T$-equivariant.

For $i=1,\dots,n$, let $\bL_i$ be the $i$-th tautological line bundle over $\Mbar_{0,n}(X, \beta)$ formed
by the cotangent line at the $i$-th marked point. Define the $i$-th descendant class $\psi_i$ as
$$
\psi_i := c_1(\bL_i)\in H^2(\Mbar_{0,n}(X, \beta);\bQ).
$$
We choose a $T$-equivariant
lift $\psi_i^{T}\in H^2_{T}(\Mbar_{0,n}(X, \beta);\bQ)$
of $\psi_i$. 

Let $\gamma_1,\dots, \gamma_n\in H_{T}^*(X,\bC)$ and $a_1,\dots,a_n\in \bZ_{\ge 0}$. We define the genus-$0$, degree-$\beta$, $T$-equivariant \emph{descendant Gromov-Witten invariant}
\begin{align*}
	&\langle \tau_{a_1}(\gamma_1), \dots, \tau_{a_n}(\gamma_n)\rangle^{X,T}_{0,\beta} \\
	& := \int_{[\Mbar_{0,n}(X, \beta)^{T}]^{\vir,T}} \frac{\iota^*\big(\prod_{i=1}^n \ev_i^*(\gamma_i)(\psi_i^{T})^{a_i}\big)}{e_{T}(N^\vir)}
	\quad \in \bC(u_1,u_2)
\end{align*}
where $\Mbar_{0,n}(X, \beta)^{T}$ is the $T$-fixed locus, $e_{T}(N^\vir)$ is the $T$-equivariant Euler class of the virtual normal bundle of $\Mbar_{0,n}(X,\beta)^{T}$ in $\Mbar_{0,n}(X,\beta)$, $[\Mbar_{0,n}(X, \beta)^{T}]^{\vir,T}$ is the virtual fundamental class, and $\iota: \Mbar_{0,n}(X, \beta)^{T}\hookrightarrow \Mbar_{0,n}(X, \beta)$ is the inclusion map.

\subsection{Graph sum formula}
\begin{definition}
	Let $n\in\bZ_{\geq 0}$ and $\beta\in E(X)$ be an effective curve class. A genus-zero, $n$-pointed, degree-$\beta$ decorated graph for $X$ is a tuple $\vGa=(\Gamma,\vf,\vd,\vs)$ where:
	\begin{itemize}
		\item $\Gamma$ is a tree. Let $V(\Gamma)$ denote the set of
		vertices in $\Gamma$, $E(\Gamma)$ denote the set of edges in $\Gamma$, and $F(\Gamma):=\{(e,v)\in E(\Gamma)\times V(\Gamma)\mid e\textrm{ is attached to } v \}$ be the set of flags.
		\item $\vf:V(\Gamma)\cup E(\Gamma)\to \{\sigma_0,\sigma_1,\sigma_2\}\cup\{\tau_1,\tau_2\} $ is the \emph{label map} that sends each vertex $v\in V(\Gamma)$ to some $\sigma_v\in \{\sigma_0,\sigma_1,\sigma_2\}$ and each edge $e\in E(\Gamma)$ to some $\tau_e\in \{\tau_1,\tau_2\}$. Let $\Upsilon_{\Sigma}$ be the graph with vertices $\sigma_0,\sigma_1,\sigma_2$ and with edges $\tau_1,\tau_2$ joining $\sigma_0,\sigma_1$ and $\sigma_0,\sigma_2$ respectively. We require that $\vf$ defines a map from
		the graph $\Gamma$ to the graph $\Upsilon_{\Sigma}$.
		\item $\vd: E(\Gamma)\to\bZ_{>0}$ is the \emph{degree map}. We denote $d_e:=\vd(e)$ for each $e\in E(\Gamma)$. We require that $\sum_{e\in E(\Gamma)}d_e[l_{\tau_e}]=\beta$.
		\item $\vs:\{1,\cdots,n\}\to V(\Gamma)$ is the \emph{marking map}.
		
	\end{itemize}
	
\end{definition}
Let $ G_{n}(X,\beta)$ denote the set of all genus-zero, $n$-pointed, degree-$\beta$ decorated graphs for $X$. Given $\vGa\in G_{n}(X,\beta)$, we introduce the following notations:
\begin{itemize}
	\item For each $v\in V(\Gamma)$, define $E_v:=\{e\in E(\Gamma)\mid (e,v)\in F(\Gamma)\}$ and $S_v:=\vs^{-1}(v)$. Let $\val(v):=|E_v|$.
	\item Let $\Aut(\vGa)$ denote the group of automorphisms of $\vGa$, i.e. automorphisms of the graph $\Gamma$ that make the label maps $\vf,\vd,\vs$ invariant.
\end{itemize}

Given $\vGa\in G_{n}(X,\beta)$, let
$$
V(\Gamma)=V_0(\vGa)\sqcup V_1(\vGa)\sqcup V_2(\vGa),
$$
where
\begin{equation}
	V_i(\vGa):=\{v\in V(\Gamma)\mid \vf(v)=\sigma_i\},\quad i=0,1,2.
\end{equation}
Let
$$
V^{1,1}_0(\vGa)\subset V_0(\vGa)
$$
be defined as
\begin{equation}
V^{1,1}_0(\vGa):=\{v\in V_0(\vGa)\mid \val(v)=2, \{\tau_{e_1},\tau_{e_2}\}=\{\tau_1,\tau_2\},d_{e_1}=d_{e_2}\},
\end{equation}
where $e_1,e_2$ are the two edges attached to $v$. 

Consider the decomposition
$$
 G_{n}(X,\beta)=\bigsqcup_{k\geq 0}G_{n}(X,\beta)^k
$$
where $G_{n}(X,\beta)^k$ is defined as
\begin{equation}
	G_{n}(X,\beta)^k:=\{\vGa\in G_{n}(X,\beta)\mid |V_0(\vGa)|-|V^{1,1}_0(\vGa)|=k\}.
\end{equation}
For $\vGa\in G_{n}(X,\beta)^1$, let $V_0(\vGa)=V^{1,1}_0(\vGa)\cup\{v_*\}$. Then we can further decompose $G_{n}(X,\beta)^1$ as
$$
G_{n}(X,\beta)^1=\bigsqcup_{l\geq 1}G_{n}(X,\beta)^{1,l}
$$
where $G_{n}(X,\beta)^{1,l}$ is defined as
\begin{equation}
	G_{n}(X,\beta)^{1,l}:=\{\vGa\in G_{n}(X,\beta)^1\mid \val(v_*)=l\}.
\end{equation}
Let $\beta=(d_1,d_2)$. Then we remark that 
\begin{equation}\label{eqn:empty}
	G_{n}(X,\beta)^0=\emptyset
\end{equation}
when $d_1\neq d_2$.

We define
\begin{equation}\label{eqn:edgeX}
	\sh(\tau_i,d)=\frac{(-1)^d d^{2d}}{(d!)^2u_i^{2d}}\prod_{j=1}^{d-1}\big(-u_1-u_2+j\frac{u_i}{d}\big)
\end{equation}
for $i=1,2$.
We define
\begin{equation}\label{eqn:vertexX1}
	\sw(\sigma_i)=-u_i(u_1+u_2)
\end{equation}
for $i=1,2$ and 
\begin{equation}\label{eqn:vertexX2}
	\sw(\sigma_0)=u_1u_2.
\end{equation}

By \cite[Theorem 73]{Liu13}, we have the following Proposition:
\begin{proposition}\label{prop:locX}
	We have the following graph sum formula：
	\begin{eqnarray*}
	&&\left\langle  \tau_{a_1}(\gamma_1),\cdots,\tau_{a_n}(\gamma_n)\right\rangle ^{X,T}_{0,\beta}\\
	&=&\sum_{\vGa\in G_{n}(X,\beta)}\frac{1}{|\Aut(\vGa)|}\prod_{e\in E(\Gamma)}\frac{\sh(\tau_e,d_e)}{d_e}\prod_{v\in V(\Gamma)}\big(\sw(\sigma_v)^{\val(v)-1}\prod_{i\in S_v}i^*_{\sigma_v}\gamma_i \big)\\
	&&\cdot \prod_{v\in V(\Gamma)}\int_{\Mbar_{0,E_{v}\cup S_{v}}}\frac{\prod_{i\in S_v}\psi_i^{a_i}}{\prod_{e\in E_{v}}(w_{(e,v)}-\psi_{(e,v)})}.
	\end{eqnarray*}
\end{proposition}
In Proposition \ref{prop:locX}, we still used \eqref{eqn:unstable} for the convention of the unstable integrals.

\section{Open/closed correspondence}\label{sec:correspondence}
Consider the graph sum formula for $\left\langle  \tau_{a_1}(\phi_{\alpha_1}),\cdots,\tau_{a_n}(\phi_{\alpha_n})\right\rangle ^{(\bP^1,L),S^1}_{\beta}$ in Proposition \ref{prop:locP1}. Given $\vf\in G_{n}(\bP^1,L,\beta)$, define   
\begin{eqnarray*}
	C_{\vGa}
	&=&\frac{1}{|\Aut(\vGa)|}\prod_{e\in E(\Gamma)\setminus\{e_0\}}\frac{\sh(\tau,d_e)}{d_e}\prod_{v\in V(\Gamma)}\big(\sw(\sigma_v)^{\val(v)-1}\prod_{i\in S_v}i^*_{\sigma_v}\phi_{\alpha_i} \big)\\
	&&\cdot \prod_{v\in V(\Gamma)}\int_{\Mbar_{0,E_{v}\cup S_{v}}}\frac{\prod_{i\in S_v}\psi_i^{a_i}}{\prod_{e\in E_{v}}(w_{(e,v)}-\psi_{(e,v)})}D_{\pm}(\mu).
\end{eqnarray*}
which is the contribution of the graph $\vGa$ in Proposition \ref{prop:locP1}. Consider the graph sum formula for $\left\langle\tau_{a_1}(\tphi_{\alpha_1}),\cdots,\tau_{a_n}(\tphi_{\alpha_n})\right\rangle ^{X,T}_{0,\beta}$ in Proposition \ref{prop:locX}. Given $\vf\in G_{n}(X,\beta)$, define
\begin{eqnarray*}
	\tC_{\vGa}
	&=&\frac{1}{|\Aut(\vGa)|}\prod_{e\in E(\Gamma)}\frac{\sh(\tau_e,d_e)}{d_e}\prod_{v\in V(\Gamma)}\big(\sw(\sigma_v)^{\val(v)-1}\prod_{i\in S_v}i^*_{\sigma_v}\tphi_{\alpha_i} \big)\\
	&&\cdot \prod_{v\in V(\Gamma)}\int_{\Mbar_{0,E_{v}\cup S_{v}}}\frac{\prod_{i\in S_v}\psi_i^{a_i}}{\prod_{e\in E_{v}}(w_{(e,v)}-\psi_{(e,v)})}.
\end{eqnarray*}
which is the contribution of the graph $\vGa$ in Proposition \ref{prop:locX}.

\begin{lemma}\label{lem:vanish}
For $\vGa\in G_{n}(X,\beta)^k$, we have 
	$$
	\tC_{\vGa}=\sum_{a\geq k-1} (u_1+u_2)^af_{a,\vGa}(u_1),
	$$
for some functions $f_{a,\vGa}(u_1)$. Here we view $\tC_{\vGa}$ as a function of $u_1$ and $u_1+u_2$.
\end{lemma}
\begin{proof}
We count the power of $(u_1+u_2)$ in $\tC_{\vGa}$. By \eqref{eqn:edgeX}, the power of $(u_1+u_2)$ in the edge factor $\sh(\tau_e,d_e)$ is 0. By \eqref{eqn:vertexX1} and \eqref{eqn:vertexX2}, when $\sigma_v=\sigma_1$ or $\sigma_v=\sigma_2$, the power of $(u_1+u_2)$ in the vertex factor $\sw(\sigma_v)^{\val(v)-1}$ is $\val(v)-1$. When $\sigma_v=\sigma_0$, the power of $(u_1+u_2)$ in the vertex factor $\sw(\sigma_v)^{\val(v)-1}$ is 0. By \eqref{eqn:tphi}, the power of $(u_1+u_2)$ in the vertex factor $i^*_{\sigma_v}\tphi_{\alpha_i}$ is 0.

For the vertex factor $\int_{\Mbar_{0,E_{v}\cup S_{v}}}\frac{\prod_{i\in S_v}\psi_i^{a_i}}{\prod_{e\in E_{v}}(w_{(e,v)}-\psi_{(e,v)})}$, it is easy to see that the power of $(u_1+u_2)$ is non-negative when $v\notin V^{1,1}_0(\vGa)$. When $v\in V^{1,1}_0(\vGa)$, the power of $(u_1+u_2)$ is -1 by \eqref{eqn:unstable}.

Therefore the power of $(u_1+u_2)$ in $\tC_{\vGa}$ is greater or equal to 
\begin{eqnarray*}
&&\sum_{v\in  V_1(\vGa)\sqcup V_2(\vGa)}(\val(v)-1)-|V^{1,1}_0(\vGa)|=|E(\vGa)|-|V_1(\vGa)|-|V_2(\vGa)|-|V^{1,1}_0(\vGa)|\\
&=&|E(\vGa)|-|V(\vGa)|+|V_0(\vGa)|-|V^{1,1}_0(\vGa)|\\
&=&-1+|V_0(\vGa)|-|V^{1,1}_0(\vGa)|,
\end{eqnarray*}
where in the last equality we use the fact that $\Gamma$ is a tree and hence $|E(\vGa)|-|V(\vGa)|+1=0$. Notice that for $\vGa\in G_{n}(X,\beta)^k$, $|V_0(\vGa)|-|V^{1,1}_0(\vGa)|=k$. The lemma follows immediately.

\end{proof}

\begin{lemma}\label{lem:G11}
	Let $\beta=(d_1,d_2),d_1\neq d_2$ and let $\mu=|d_1-d_2|$. Then we have 
	$$
	\left\langle  \tau_{a_1}(\phi_{\alpha_1}),\cdots,\tau_{a_n}(\phi_{\alpha_n})\right\rangle ^{(\bP^1,L),S^1}_{(0,1),\beta}=(-1)^{\mu+1}\mu\sum_{\vGa\in G_{n}(X,\beta)^{1,1}}\tC_{\vGa}\vert_{u_1+u_2=0,u_1=u}.
	$$
\end{lemma}
\begin{proof}
	We prove the lemma by comparing the graph sum formulae in Proposition \ref{prop:locP1} and Proposition \ref{prop:locX}. We first construct a bijection 
	$$
	\varphi:G_{n}(X,\beta)^{1,1}\to G_{n}(\bP^1,L,\beta)
	$$
	as follows. Given $\vGa\in G_{n}(X,\beta)^{1,1}$, we remove all the vertices in $V^{1,1}_0(\vGa)$ and combine the edges $e_1,e_2$ into one edge $e$ for any $e_1,e_2$ attaching to a vertex $v\in V^{1,1}_0(\vGa)$. The degree $d_e$ is defined to be $d_{e_1}=d_{e_2}$. The label $\sigma_v$ for $v\in V_1(\vGa)$ is replaced by $\sigma_+$ and the label $\sigma_v$ for $v\in V_2(\vGa)$ is replaced by $\sigma_-$. The label of the unique vertex $v_*\in V_0(\vGa)\setminus V^{1,1}_0(\vGa)$ is removed. The resulting decorated graph is denoted by $\varphi(\vGa)$. The unique vertex $v_*\in V_0(\vGa)\setminus V^{1,1}_0(\vGa)$ corresponds to the root in $\varphi(\vGa)$. Also notice that the degree of the unique edge $e$ attaching to $v_*$ is $\mu$. It is easy to see that $\varphi$ is a bijection.
	
	Now let us compare the contributions $\tC_{\vGa}$ and $C_{\varphi(\vGa)}$. First we cancel the powers of $-u_1-u_2$ coming from the factors $\sw(\sigma_v)^{\val(v)-1},v\in V_1(\vGa)\sqcup V_2(\vGa)$ and from the factors $\int_{\Mbar_{0,E_{v}\cup S_{v}}}\frac{\prod_{i\in S_v}\psi_i^{a_i}}{\prod_{e\in E_{v}}(w_{(e,v)}-\psi_{(e,v)})}, v\in V^{1,1}_0(\vGa)$ in $\tC_{\vGa}$ (see Lemma \ref{lem:vanish}). Then we apply the restriction $u_1+u_2=0$ to each factor in $\tC_{\vGa}$. By Proposition \ref{prop:locP1} and Proposition \ref{prop:locX}, it is easy to check that $C_{\varphi(\vGa)}=(-1)^{\mu+1}\mu\tC_{\vGa}$. The lemma then follows from Proposition \ref{prop:locP1}. 
\end{proof}

\begin{theorem}\label{thm:correspondence}
	Let $\beta=(d_1,d_2)\in\bZ_{\geq 0}^2,d_1\neq d_2$. Then for $\alpha_1,\cdots,\alpha_n\in\{1,2\},a_1,\cdots,a_n\in\bZ_{\geq 0}$, we have
	\begin{eqnarray*}
		&&\left\langle  \tau_{a_1}(\tphi_{\alpha_1}),\cdots,\tau_{a_n}(\tphi_{\alpha_n})\right\rangle ^{X,T}_{0,\beta}\vert_{u_1+u_2=0,u_1=u}\\
		&=&\sum \frac{(-u^2)^{l+m-1}}{|\Aut(\mu^1,d^1,A^1)||\Aut(\mu^2,d^2,A^2)|}\prod_{i=1}^{l}\big(\frac{\mu^1_i}{-u}\big)^{1+b_i}\prod_{j=1}^{m}\big(\frac{\mu^2_j}{u}\big)^{1+c_j}\frac{(l+m-3)!}{\prod_{i=1}^{l}b_i!\prod_{j=1}^{m}c_j!}\\
		&&\cdot\prod_{i=1}^{l}\big(\frac{(-1)^{\mu^1_i}}{u}\left\langle  \prod_{k\in A^1_i}\tau_{a_k}(\phi_{\alpha_k})\right\rangle ^{(\bP^1,L),S^1}_{(0,1),\beta^1_i}\big)
		\prod_{j=1}^{m}\big(\frac{(-1)^{\mu^2_j+1}}{u}\left\langle  \prod_{k\in A^2_j}\tau_{a_k}(\phi_{\alpha_k})\right\rangle ^{(\bP^1,L),S^1}_{(0,1),\beta^2_j}\big).
	\end{eqnarray*}
	Here the sum is taken over $\{((\mu^1_i,d^1_i,A^1_i,b_i)_{1\leq i\leq l},(\mu^2_j,d^2_j,A^2_j,c_j)_{1\leq j\leq m},)\mid l,m\geq 0, l+m\geq 1,\mu^1_1\neq \mu^2_1 \textrm{ when } l=m=1, \mu^1_i,\mu^2_j>0, d^1_i,d^2_j,b_i,c_j\geq 0, \sum_{i=1}^{l}(d^1_i+\mu^1_i)+\sum_{j=1}^{m}d^2_j=d_1, \sum_{i=1}^{l}d^1_i+\sum_{j=1}^{m}(d^2_j+\mu^2_j)=d_2,\sqcup_{i=1}^{l}A^1_i\sqcup \sqcup_{j=1}^{m}A^2_j=\{1,\cdots,n\},\sum_{i=1}^{l}b_i+\sum_{j=1}^{m}c_j=l+m-3\}$ and $ \beta^1_i=(d^1_i+\mu^1_i,d^1_i), \beta^2_j=(d^2_j,d^2_j+\mu^2_j)$. The automorphism groups $\Aut(\mu^1,d^1,A^1)$ and $\Aut(\mu^2,d^2,A^2)$ are for the tuples $\big((\mu^1,d^1,A^1)_{1\leq i\leq l} \big)$ and $\big((\mu^2,d^2,A^2)_{1\leq j\leq m} \big)$ respectively.
	
\end{theorem}

\begin{proof}
	By Lemma \ref{lem:vanish} and \eqref{eqn:empty}, we have
	$$
	\left\langle  \tau_{a_1}(\tphi_{\alpha_1}),\cdots,\tau_{a_n}(\tphi_{\alpha_n})\right\rangle ^{X,T}_{0,\beta}\vert_{u_1+u_2=0,u_1=u}=\sum_{\vGa\in G_{n}(X,\beta)^{1}}\tC_{\vGa}\vert_{u_1+u_2=0,u_1=u}.
	$$
	Given $\vGa\in G_{n}(X,\beta)^{1}$, let $v_*$ be the unique vertex in $ V_0(\vGa)\setminus V^{1,1}_0(\vGa)$. Let $e^1_1,\cdots,e^1_l,e^2_1,\cdots,e^2_m$ be the edges attaching to $v_*$, where $\vf(e^1_i)=\tau_1,\vf(e^2_j)=\tau_2,\vd(e^1_i)=\mu^1_i,\vd(e^2_j)=\mu^2_j$. Remove $v_*$ and attach a univalent vertex $v^k_i$ to each $e^k_i$ and define $\vf(v^k_i)=\sigma_0$, where $(k,i)\in\{(1,1),\cdots,(1,l),(2,1),\cdots,(2,m)\}$. Since $\vGa$ is a tree, the resulting graph has $l+m$ connected components $\vGa^k_i,(k,i)\in\{(1,1),\cdots,(1,l),(2,1),\cdots,(2,m)\}$ and each $\vGa^k_i$ lies in $G_{|A^k_i|}(X,\beta^k_i)^{1,1}$ for some $A^k_i,\beta^k_i$ with $\beta^1_i=(d^1_i+\mu^1_i,d^1_i), \beta^2_j=(d^2_j,d^2_j+\mu^2_j)$ and $\sqcup_{i=1}^{l}A^1_i\sqcup \sqcup_{j=1}^{m}A^2_j=\{1,\cdots,n\}$. Therefore we have
		\begin{eqnarray*}
	&&\tC_{\vGa}\vert_{u_1+u_2=0,u_1=u}\\
	&=&\frac{(-u^2)^{l+m-1}}{|\Aut(\mu^1,d^1,A^1)||\Aut(\mu^2,d^2,A^2)|}\int_{\Mbar_{0,l+m}}\frac{1}{\prod_{i=1}^{l}(\frac{-u}{\mu^1_i}-\psi_{i})\prod_{j=1}^{m}(\frac{u}{\mu^2_j}-\psi_{l+j})}\\
	&\cdot&\prod_{i=1}^{l}(\frac{\mu^1_i}{-u_1}\tC_{\vGa^1_i}\vert_{u_1+u_2=0,u_1=u})\prod_{j=1}^{m}(\frac{\mu^2_j}{u_1}\tC_{\vGa^2_j}\vert_{u_1+u_2=0,u_1=u})
		\end{eqnarray*}
By Lemma \ref{lem:G11} and the fact that
$$
\int_{\Mbar_{0,h}}\psi_1^{s_1}\cdots \psi_h^{s_h}=\frac{(h-3)!}{\prod_{i=1}^{h}s_i!},\quad s_1+\cdots+s_h=h-3,
$$
the theorem follows immediately.
\end{proof}

\end{CJK}
\end{document}